\documentclass[12pt]{amsart}
\usepackage{amssymb,amsmath,amsthm,amsfonts,mathrsfs}
\usepackage[hmargin=3cm,vmargin=3.5cm]{geometry}
\usepackage{hyperref}
\usepackage{color}
\usepackage{graphicx,psfrag}
\usepackage{amscd}  
\usepackage[shellescape]{gmp}
\usepackage{stmaryrd}  %% double square brackets 
\usepackage[all,2cell]{xy} \UseAllTwocells \SilentMatrices
\usepackage{tikz-cd}
\usepackage[dvips]{epsfig}
\usepackage{MnSymbol}
\usepackage{enumitem}
\usepackage{comment}

\newtheorem{theorem}{Theorem}[section]
\newtheorem{lemma}[theorem]{Lemma}

\newtheorem{definition}[theorem]{Definition}

\newtheorem{cor}[theorem]{Corollary}

\theoremstyle{remark}
\newtheorem{remark}[theorem]{Remark}

\usepackage{todonotes}

\let\oldtocsection=\tocsection
\let\oldtocsubsection=\tocsubsection
\renewcommand{\tocsection}[2]{\hspace{0em}\oldtocsection{#1}{#2}}
\renewcommand{\tocsubsection}[2]{\hspace{1em}\oldtocsubsection{#1}{#2}}

\usepackage{chngcntr}
\title[On the universal pairing for \texorpdfstring{$2$}{2}-complexes]{On the universal pairing for \texorpdfstring{$2$}{2}-complexes}

\author{Mikhail Khovanov}
 \address{Department of Mathematics, Columbia University, New York, NY 10027, USA \newline  
  \phantom{} \hspace{0.06in} Department of Mathematics, Johns Hopkins University, Baltimore, MD,  21218, USA}
 \email{\href{mailto:khovanov@math.columbia.edu}{khovanov@math.columbia.edu}}

\author{Vyacheslav Krushkal}
\address{Department of Mathematics, University of Virginia, Charlottesville, VA 22904}
\email{\href{mailto:krushkal@virginia.edu}{krushkal@virginia.edu}}

\author{John Nicholson}
\address{School of Mathematics and Statistics, University of Glasgow, United Kingdom}
\email{\href{mailto:john.nicholson@glasgow.ac.uk}{john.nicholson@glasgow.ac.uk}}

\date{February 12, 2024}

\begin{document}

\def\R{\mathbb R}
\def\Q{\mathbb Q}
\def\Z{\mathbb Z}
\def\N{\mathbb N} 
\def\C{\mathbb C}
\def\S{\mathbb S}
\def\CP{\mathbb P}
\renewcommand\SS{\ensuremath{\mathbb{S}}}
\newcommand{\ophana}{\overline{\phantom{a}}}
\newcommand{\one}{\mathbf{1}} %identity object
\newcommand{\mcK}{\mathcal{K}}
\newcommand{\mcU}{\mathcal{U}}
\newcommand{\mcA}{{\mathcal A}}
\newcommand{\mcF}{\mathcal{F}}

\def\l{\lbrace}
\def\r{\rbrace}
\def\o{\otimes}
\def\lra{\longrightarrow}
\def\Ext{\mathrm{Ext}}
\def\mc{\mathcal}
\def\mf{\mathfrak} 
\def\mcC{\mathcal{C}}
\def\mcI{\mathcal{I}}
\def\uFr{\underline{\mathrm{Fr}}}

\def\Bool{\mathbb{B}}
\def\ovb{\overline{b}}
\def\tr{{\sf tr}} 
\def\det{{\sf det }} 

\def\lra{\longrightarrow}
\def\kk{\mathbf{k}}  %% base field  
\def\gdim{\mathrm{gdim}}  %% graded dimension 
\def\rk{\mathrm{rk}}
\def\undep{\underline{\epsilon}}

% cobordism categories 
\def\CCC{\mathcal{C}} % cat of cobordisms 
\def\kCCC{\kk\CCC}  % prelinearization of tfs 
\def\udCCC{\underline{\dCCC}}  %  quotient of Deligne by negligibles

\def\Cob{\mathrm{Cob}} 
\def\Cobtwo{\Cob_2}   %  2D cobordisms
\def\Kob{\mathrm{Kob}}
\def\Kar{\mathrm{Kar}}   % Karoubi envelope 

\def\dmod{\mathrm{-mod}}   % modules  

\newcommand{\alphai}{\alpha^{\bullet}}

\newcommand{\oplusop}[1]{{\mathop{\oplus}\limits_{#1}}}
\newcommand{\ang}[1]{\langle #1 \rangle }

\newcommand{\undn}{\mathbf{n}}

\newcommand{\Hom}{\mbox{Hom}}
\newcommand{\id}{\mbox{id}}
\newcommand{\Id}{\mbox{Id}}
\newcommand{\End}{\mathrm{End}}
%\newcommand{\iHom}{\underline{\mbox{Hom}}}

%\newcommand{\drawing}[1]{
%\begin{center}{\psfig{figure=fig/#1}}\end{center}}

%% creating long squiggly arrow \xrsquigarrow{}  
\usetikzlibrary{calc,decorations.pathmorphing,shapes}

\newcounter{sarrow}
\newcommand\xrsquigarrow[1]{
\stepcounter{sarrow}
\mathrel{\begin{tikzpicture}[baseline= {( $ (current bounding box.south) + (0,-0.5ex) $ )}]
\node[inner sep=.5ex] (\thesarrow) {$\scriptstyle #1$};
\path[draw,<-,decorate,
  decoration={zigzag,amplitude=0.7pt,segment length=1.2mm,pre=lineto,pre length=4pt}] 
    (\thesarrow.south east) -- (\thesarrow.south west);
\end{tikzpicture}}
}

\begin{abstract} The universal pairing for manifolds was defined and shown to lack positivity in dimension $4$ in \cite{FKNSWW}. We prove an analogous result for $2$-complexes, and show that the universal pairing does not detect the difference between simple homotopy equivalence and $3$-deformations. The question of whether these two equivalence relations are different for $2$-complexes is the subject of the Andrews-Curtis conjecture.
We also discuss the universal pairing for higher-dimensional complexes and show that it is not positive.
\end{abstract}

\maketitle

\setcounter{tocdepth}{1}
\tableofcontents

%%%%%%%%%%%%%%%%%%%%%%%
%
%  Intro  
%
%%%%%%%%%%%%%%%%%%%%%%%

\section{Introduction}

The universal pairing for manifolds was introduced in \cite{FKNSWW}. For an oriented compact $(d-1)$-manifold $S$, ${\mathcal M}_S$ is the vector space of formal ${\mathbb C}$-linear combinations of oriented $d$-manifolds with boundary equal to $S$. 
The universal pairing $\langle\; , \; \rangle\colon {\mathcal M}_S\times {\mathcal M}_S\longrightarrow {\mathcal M}_{\emptyset}$, is defined by gluing manifolds along their common boundary (reversing the orientation and conjugating the coefficients of the second argument). The pairing is {\em positive} if $\langle x, x\rangle=0$ implies $x=0$ for all $x\in {\mathcal M}_S$. 

The pairing is positive in dimensions $d\leq 3$ \cite{CFW, FKNSWW}. 
In dimension $4$ the universal pairing is not positive. Given a cork (certain compact, smooth, contractible $4$-manifold with boundary) $M$, with an involution $\tau$ on the boundary, let $x=M-M'$. Here $M'$ is the same manifold, where the identification of the boundary with $S=\partial M$ is twisted by $\tau$. Then $\langle x, x\rangle=0$, and for some corks $x\neq 0\in {\mathcal M}_S$. The structure of corks arising in $h$-cobordisms of $4$-manifolds was used in \cite{FKNSWW} to show that unitary $(3+1)$-dimensional TQFTs cannot distinguish exotic smooth structures on simply-connected $4$-manifolds.  A follow-up work \cite{Re} showed, using different methods, that more generally, semi-simple TQFTs cannot distinguish exotic smooth structures on simply-connected $4$-manifolds. 
Positivity in higher dimensions, $d\geq 5$ was analyzed in \cite{KT}.

The Andrews-Curtis conjecture concerns the difference between $2$-complexes up to simple homotopy equivalence and $2$-complexes up to {\em $3$-deformations}, where elementary expansions and collapses are allowed with cells only up to dimension $3$. This is a formulation in terms of $2$-complexes of the group-theoretic conjecture which is often phrased using group presentations, cf. \cite{HM1} and references therein. 
Note that in dimensions $n\geq 3$, two $n$-complexes are simple homotopy equivalent if and only if they are related by an $(n+1)$-deformation \cite{Wa1}.

There is a well-known observation, cf. \cite{Quinn1}, about the analogy between the ``nilpotent'' stabilization property of exotic smooth structures on simply-connected $4$-manifolds and of the Andrews-Curtis conjecture. Given any two closed smooth simply-connected, homeomorphic $4$-manifolds $M_1, M_2$, manifolds $M_1\#^n S^2\times S^2$ and $M_2\#^n S^2\times S^2$ are diffeomorphic, for some~$n$. Similarly, given two simple-homotopy equivalent $2$-complexes $K_1, K_2$, complexes $K_1\vee^n S^2$ and $K_2\vee^n S^2$ are $3$-deformation equivalent,  for some $n$.
We extend this analogy to universal pairings, by establishing a result similar to that of \cite{FKNSWW} for pairings of 2-complexes.

\vspace{0.05in} 

The framework for the universal pairings of $2$-complexes is set up in Section \ref{sec: universal construction for 2-complexes}. We consider $2$-complexes up to $3$-deformations, that is, up to Andrews-Curtis equivalence.
The fact that the pairing is not positive is proved in Corollary \ref{cor: null vector not homotopic}.
Our proof relies on 
a general statement, established in Theorem \ref{thm: simple homotopy versus 3 deformations}, that 
given a pair $L_1, L_2$ of $2$-complexes with isomorphic fundamental groups and the same Euler characteristic, they are $3$-deformation equivalent to $2$-complexes $L'_1, L'_2$ whose difference is a null vector for the universal pairing.

We note that the universal pairing in the lower dimensional case, for graphs, is positive. This was proved in \cite[Theorem 6.6]{CFW} using graph tensor TQFTs. However, there are versions incorporating additional decorations and symmetries of graphs for which the pairing is not positive \cite{Cl}.

To focus on simple homotopy equivalence, in Section
\ref{sec: simple homotopy equivalent} we define a version of the universal pairing using an equivalence relation that is finer than $3$-deformation. Theorem \ref{thm: simple homotopy versus 3 deformations+} shows that this version of the universal pairing also cannot detect the potential difference between simple homotopy equivalence and $3$-deformations (or in other words, a potential counterexample to the Andrews-Curtis conjecture). The proof relies on 
a result of Quinn \cite{Quinn} that simple homotopy equivalent $2$-complexes are related by $3$-deformations and an {\em $s$-move}.
Thus, the $s$-move plays a role for $2$-complexes which is analogous to the role of corks for smooth $4$-manifolds in the proof in \cite{FKNSWW}.

Finally, in Section \ref{sec: higher dimensions} we consider the setting of higher dimensional complexes and show that the universal pairing lacks positivity in this case as well.

\subsection*{Acknowledgments}
We would like to thank Mike Freedman for discussions on the universal pairing of manifolds and complexes.

M.K. was partially supported by NSF grant DMS-2204033 and Simons Collaboration Award 994328 while working on this paper. V.K. was supported in part by NSF Grant DMS-2105467.

%%%%%%%%%%%%%%%%%%%%%
%
% 2-complexes 
%
%%%%%%%%%%%%%%%%%%%%%

%\section{\texorpdfstring{$2$}{2}-complexes} \label{sec: 2-complexes}

\section{2-complexes} \label{sec: background}

All manifolds considered in this paper are compact and all CW complexes are assumed to be finite.

We start with a brief discussion of group presentations and of the corresponding equivalence relations on $2$-complexes, referring the reader to \cite{HM1} for more details. 

Let $F_{g_1, \ldots, g_m} $ denote the free group generated by $g_1,\ldots, g_m$. Given a group presentation $$G\cong \langle g_1\ldots, g_m|R_1, \ldots, R_n\rangle, $$ consider the following transformations: 
\begin{enumerate}[label=(\roman*), itemsep=1mm]
    \item $R_j\mapsto w R_j w^{-1}, \, R_j\mapsto R_j^{-1}, \, R_j\mapsto R_j R_k$ or $ R_j\mapsto R_k R_j$ where  $w\in F_{g_1, \ldots, g_m}$, $j\neq k$. 
    \item Nielsen transformations on $F_{g_1, \ldots, g_m} $: $g_i\mapsto g_i^{-1}, \, g_i\mapsto g_i g_j$, $ g_i\mapsto g_j g_i$, $i\neq j$.
    \item $\langle g_1\ldots, g_m|R_1, \ldots, R_n\rangle \mapsto \langle g_1\ldots, g_m, g_{m+1}|R_1, \ldots, R_n, g_{m+1}\rangle$ and its inverse,
    \item $\langle g_1\ldots, g_m|R_1, \ldots, R_n\rangle \mapsto \langle g_1\ldots, g_m|R_1, \ldots, R_n, 1\rangle$ and its inverse.
\end{enumerate}

The entire collection of transformations (i)-(iv) is equivalent to Tietze moves.
Therefore two presentations give isomorphic groups if and only if they are related by transformations (i)-(iv). 

Consider the standard $2$-complex associated with a presentation as above, with a single $0$-cell, $m$ $1$-cells corresponding to the generators $g_i$ and $n$ $2$-cells corresponding to the relations $R_j$. The equivalence classes of group presentations with respect to transformations (i)-(iii) are in bijective correspondence with {\em $3$-deformation} types of $2$-complexes, cf. \cite[Theorem 2.4]{HM1}, where $3$-deformations are compositions of elementary expansions and collapses involving cells of dimension at most $3$. (Some authors, for example \cite{Quinn}, refer to this equivalence relation as a $2$-deformation. We follow the convention in \cite{HM1}.) It follows from the Tietze theorem that two $2$-complexes $K, L$ have isomorphic fundamental groups if and only if $K\vee^k S^2$ is $3$-deformation equivalent to $L\vee^{\ell} S^2$ for some $k, \ell$. It is worth noting that the transformations (ii), (iii) correspond to 2-deformations.

Thus, there are several equivalence relations on compact $2$-complexes $K, L$ that are of interest:
\begin{enumerate}
    \item  Stable equivalence: $K \vee^n \SS^2$ is homotopy equivalent to $L \vee^n \SS^2$ for some $n$. By the discussion above, this is equivalent to $K \vee^n \SS^2$ having the same $3$-deformation type as $L \vee^n \SS^2$ for some $n$, and also equivalent to $\pi_1(K)\cong \pi_1(L)$ and $\chi(K)=\chi(L)$. 
    \item  Homotopy equivalence.
    \item Simple homotopy equivalence.
    \item $3$-deformation: $K, L$ are related by a sequence of elementary expansions and collapses where the maximal dimension of cells is $3$.
    \item Combinatorial isomorphism.
\end{enumerate}

\begin{remark} \label{rem: equivalence}
    \hfill
\begin{enumerate}[label=(\roman*)] \label{rem: 2-complexes}
    \item A more general form of stable equivalence: $K \vee^k \SS^2$ is homotopy equivalent to $L \vee^{\ell} \SS^2$ for some $k, \ell$ is equivalent to $\pi_1(K)\cong \pi_1(L)$. However it is natural to impose $\chi(K)=\chi(L)$ as well  
    so that $K, L$ are not trivially distinguished by the Euler characteristic.
  \item For simply-connected complexes, $(1) \Leftrightarrow(2) \Leftrightarrow (3)$. Here the second equivalence holds because the Whitehead group of the trivial group is trivial~\cite{Co,Tu}. 
    \item (3) allows elementary expansions and collapses using cells of any dimension (in fact, up to dimension 4 suffices). The elementary steps in a simple homotopy equivalence may be reordered to be ``self-indexed'', so that first all elementary expansions take place in the order of increasing dimensions, and then the elementary collapses follow in the order of decreasing dimensions, cf. \cite[(14)]{HM1}.
    \item 
    Any two simple homotopy equivalent complexes are related by a 3-deformation followed by Quinn $s$-moves~\cite{Quinn} followed by a 3-deformation, see also~\cite{Bo} and Section~\ref{sec:shequivalence_versus_3deformations} below. 
    \item Instead of general $2$-complexes, one can study special polyhedra~\cite{Ma}. 
\end{enumerate}
\end{remark}

These equivalence relations are related by $(5) \Rightarrow (4) \Rightarrow (3) \Rightarrow (2) \Rightarrow (1)$. 
We will now discuss the extent to which these arrows can be reversed.

(1) vs. (2). Examples of $2$-complexes which are stably homotopy equivalent but not homotopy equivalent are known to exist only for a limited range of fundamental groups $G$.
The following gives a summary of the known examples.
Note that the first examples were found independently by Dunwoody \cite{Dunwoody} and Metzler \cite{Metzler} in 1976.
\begin{itemize}
\item 
For finite groups, examples are known for certain abelian groups by Metzler \cite{Metzler} and $Q_{28}$ the quaternion group of order 28 by Mannan-Popiel \cite{MP}.
\item 
Lustig \cite{Lustig} constructed infinitely many 2-complexes with the same $\pi_1,\chi$ which are pairwise homotopically distinct, over $\pi_1=\langle r,s,t \mid s^2=t^3, [r^2,s], [r^2,t] \rangle$.
\item 
For the trefoil group $T=\langle x, y \mid x^2=y^3 \rangle$, examples were constructed by Dunwoody \cite{Dunwoody}, and this is extended to infinitely many examples by Harlander-Jenson \cite{HJ1} which are pairwise homotopically distinct. This was used by Nicholson \cite{N21} to construct examples with arbitrary deficiency below the maximal value over free products $T \ast \cdots \ast T$.
\item 
For the Klein bottle group $K = \langle x, y \mid y^{-1}xyx \rangle$, examples were constructed by Mannan \cite{M}.
\end{itemize}

(2) vs. (3).
Examples of $2$-complexes which are homotopy equivalent but not simple homotopy equivalent proved difficult to find. The question of whether such examples existed appeared in C. T. C. Wall's 1979 Problem's List \cite[Problem D6]{Wa}, and it took until the 1990s for examples to be found by W. Metzler and M. Lustig \cite{Lustig0, Metzler1}. 
It is worth noting that it also took until recently for homotopy equivalent but not simple homotopy equivalent closed topological 4-manifolds  
to be found \cite{NNP}, and examples are not currently known for closed smooth 4-manifolds.
One reason that finding such examples is difficult is that, in order to show that a pair of homotopy equivalent CW-complexes $X$ and $Y$ are not simple homotopy equivalent, it does not suffice to show that a given homotopy equivalence $f: X \to Y$ is not simple. One needs to show that all such $f$ are not simple, and this can be done by considering the homotopy automorphisms of $X$ (or equivalently $Y$) which can be hard to compute (see \cite[p2]{NNP}).

(3) vs. (4).
The Andrews-Curtis conjecture claims that any such simple homotopy equivalence may be implemented with expansions and collapses of cells of dimensions at most $3$, i.e. a $3$-deformation, when $\pi_1$ is trivial. Thus, the conjecture asks whether $(3) \Leftrightarrow (4)$. It is often considered in the case of trivial $\pi_1$. Considerable effort has gone into various approaches to the Andrews-Curtis conjecture. We refer the reader to \cite{BD,Quinnlect} and references therein for approaches to the conjecture using ideas from Topological Quantum Field Theory.

\section{The universal pairing for \texorpdfstring{$2$}{2}-complexes} \label{sec: universal construction for 2-complexes}
One may consider the universal pairing and the positivity problem for $2$-complexes up to each equivalence relation. 
To focus on topological aspects,
and on the relation with the Andrews-Curtis conjecture, we will consider the universal pairing for $2$-complexes up to $3$-deformations.

\begin{definition} \label{def: kn} {\rm 
For each $n\in\Z_+$ consider $\mcK_n$, the set of  equivalence classes of $2$-complexes with a subgraph in its $1$-skeleton identified with $\vee^n\SS^1$. Here we consider complexes {\em up to $3$-deformations, restricting to the identity on $\vee^n\SS^1$.}
The circle wedge summands are ordered. $\mcK_n$ is a commutative monoid, where multiplication identifies $\vee^n\SS^1$ in the two factors. Given $K\in\mcK_n$, we will refer to the specified subgraph $\vee^n\SS^1$ of its $1$-skeleton as the {\em boundary} of $K$.}
\end{definition}

Thus, $\mcK_n$ comes with a commutative associative multiplication 
\begin{equation}\label{eq_mult_K}\cdot \ : \ \mcK_n\times \mcK_n\lra \mcK_n
\end{equation}
given by taking the union of two 2-complexes along the common boundary $\vee^n\SS^1$. The 2-complex $\vee^n\SS^1$ without 2-cells is the unit element for  multiplication. There is also the forgetful map $\mcK_n \lra \mcK_0$
and the composition 
\begin{equation}
    \langle \, , \rangle \ : \ \mcK_n\times \mcK_n \stackrel{\cdot}{\lra} \mcK_n \lra \mcK_0. 
\end{equation}
Note that the group $\mathsf{Aut}(F_n)$ of automorphisms of the free group  acts on  $\mcK_n$.

Fix a commutative ring $\kk$, and  consider free $\kk$-module $\kk \mcK_n$ with the basis $\mcK_n$.
Linearizing the map \eqref{eq_mult_K} turns 
$\kk\mcK_n$ into a commutative associative $\kk$-algebra, with the multiplication 
\begin{equation} \label{eq: commutative multiplication on Kn} \cdot\,\; \colon \; \kk \mcK_n\times \kk \mcK_n\longrightarrow \kk \mcK_n,
\end{equation}
and the pairing
\begin{equation} \label{eq: universal pairing on Kn}\langle\, , \, \rangle\,\; \colon \; \kk \mcK_n\times \kk \mcK_n\longrightarrow \kk \mcK_0
\end{equation}
taking values in commutative $\kk$-algebra $\kk\mcK_0$. The product in the latter is given by the disjoint union of complexes and extended by $\kk$-linearity. 

The pairing \eqref{eq: universal pairing on Kn} is the composition of the product \eqref{eq: commutative multiplication on Kn} with the evaluation $\kk \mcK_n\longrightarrow \kk \mcK_0$ (forgetting the boundary graph), so we can write: 
\begin{equation}\label{eq_dot_bracket}
  x, y \ \longmapsto x\cdot y
  \xrightarrow[]{\mathrm{forget\ boundary}}
  \langle x,y\rangle . 
\end{equation}

The map \eqref{eq: universal pairing on Kn}
is similar to the universal pairing for manifolds.
Note that the product \eqref{eq: commutative multiplication on Kn} is unavailable for manifolds; it is a feature present in the setting of CW or simplicial complexes.

\begin{remark} {\rm Associative commutative product \eqref{eq_mult_K} of 2-complexes (union, with the common boundary identified), explored in the present paper, results in a rigid structure upon linearization. This sort of multiplication is only possible with singular structures, such as simplicial or CW-complexes and graphs. It is not present in the categories of cobordisms between manifolds, with few exceptions. Coupled to reflection positivity~\cite{FLS}, this multiplication, in the case of graphs up to isomorphism rather than 2-complexes, leads to the state spaces associated to boundaries being commutative semisimple algebras and to a classification of suitable graph evaluations via homomorphisms into weighted graphs~\cite{FLS}. }   
\end{remark} 

\begin{remark}{\rm A variation on the universal pairing due to Freedman et al. is known as the universal construction of topological theories~\cite{BHMV,Kh2}, which has found uses in link homology~\cite{Kh1,RW} and in studies of topological theories~\cite{KKO}, including those with defects~\cite{IK}.}  
\end{remark}

\section{The universal pairing, stable homotopy equivalence and \texorpdfstring{$3$}{3}-deformations} \label{sec:shequivalence_versus_3deformations}
The following theorem may be viewed as an analogue for $2$-complexes of \cite[Theorem 4.1]{FKNSWW}.
In this theorem we use the convention that a $2$-complex $K$ is considered an element of $\mcK_n$ where $n$ is the first Betti number of the $1$-skeleton of $K$. That is, in the context of Definition \ref{def: kn}, the entire $1$-skeleton $K^1$ of $K$ is identified with $\vee^n\SS^1$.

Recall the difference between the dot product $\cdot$ and the pairing $\langle\, , \, \rangle$, see~\eqref{eq_dot_bracket}.

\begin{theorem}\label{thm: simple homotopy versus 3 deformations} 
Let $L_1, L_2$ be $2$-complexes such that $\pi_1(L_1) \cong \pi_1(L_2)$ and $\chi(L_1)=\chi(L_2)$. Then there exist $n$ and $L_1',L_2' \in \mcK_n$ such that 
\begin{enumerate}[label= \normalfont (\roman*)]
\item 
For $i =1, 2$, $L_i$ is $2$-deformation equivalent to $L_i'$.
\item 
$x:=L_1'-L_2'\in \kk \mcK_n$ satisfies $x\cdot x=0$, so in particular $x$ is a null vector for the universal pairing: $\langle x, x\rangle=0$.
\end{enumerate}
\end{theorem}

It follows that a unitary TQFT cannot distinguish between stable equivalence
and $3$-deformation of $2$-complexes. Indeed, in a unitary TQFT, with $\kk=\C$, the null-vector property $\langle x, x\rangle=0$ implies $x=0$. This corollary, in the setting of semi-simple TQFTs over an algebraically closed field $\kk$,  can also be established using methods analogous to \cite{Re}.

\begin{proof}[Proof of Theorem \ref{thm: simple homotopy versus 3 deformations}]
We begin by making the following elementary observation, which applies even in the case where $\chi(L_1) \ne \chi(L_2)$.
A similar result is proven in \cite[Proposition 2.1]{Johnson1} but with the complexes involved taken up to homotopy equivalence rather than $3$-deformation equivalence. Our proof below follows the same argument with several modifications made throughout.
\begin{lemma} \label{lem: can fix boundary} 
Let $L_1, L_2$ be $2$-complexes with $\pi_1(L_1) \cong \pi_1(L_2)$. Then, for some $n$,  there exist $2$-complexes with fixed $\vee^n\SS^1$ boundary $L_1', L_2' \in \mcK_n$ such that
\begin{enumerate}[label= \normalfont (\roman*)]
\item 
For $i =1, 2$, $L_i$ is $2$-deformation equivalent to $L_i'$. 
\item 
There is a map $f\colon L_1'\to L_2'$ fixing their common boundary $\vee^n\SS^1= (L_1')^1= (L_2')^1$ and inducing an isomorphism on $\pi_1$. 
\end{enumerate}
\end{lemma}

\begin{proof}[Proof of Lemma \ref{lem: can fix boundary}] Let $G = \pi_1(L_1) \cong \pi_1(L_2)$. By the comments in Section \ref{sec: background}, we can alter $L_1$ and $L_2$ by $2$-deformations so that they each have a single $0$-cell and so correspond to group presentations for $G$. We will denote the presentations respectively by
$$\langle x_1,\ldots, x_a \, |\,  R_1, \ldots, R_b\rangle\; ,\;   \langle y_1,\ldots, y_c \, | \, S_1, \ldots, S_d\rangle. $$

Next view each $y_i \in F_{x_1,\ldots,x_a}$. By applying the Tietze move (iii) followed by a sequence of Tietze moves (ii), as defined in Section \ref{sec: background}, we obtain a 2-deformation from $L_1$ to 
$$\langle x_1,\ldots, x_a,x_{a+1}\, |\,  R_1, \ldots, R_b,x_{a+1}y_1^{-1}\rangle.$$
By repeating these operations on $L_1$ for $y_1,\ldots,y_c$, and applying the analogous operations on $L_2$ for $x_1,\ldots,x_a$, we get that $L_1$ and $L_2$ are 3-deformation equivalent to
\begin{align*}
    &\langle x_1, \ldots, x_{a+c} \, |\,  R_1, \ldots, R_b,x_{a+1}y_1^{-1}, \ldots, x_{a+c}y_c^{-1}\rangle, \\
    &\langle y_1,\ldots, y_{c+a} \, | \, S_1, \ldots, S_d,y_{c+1}x_1^{-1},\ldots,y_{c+a}x_a^{-1}\rangle
\end{align*}
respectively, and we will denote those 2-complexes by $L_1', L_2'$.

Let $n =a+c$. By construction, $\{x_1,\ldots,x_{a+c}\}$ and $\{y_1,\ldots,y_{c+a}\}$ are the same elements of $G$ up to permutations. Fix identifications of $L_1^1$ and $L_2^1$ with $\vee^n\SS^1$ such that each copy of $\SS^1$ corresponds to the same generator. Thus we have $L_1', L_2' \in \mcK_n$.
The identification $(L_1')^1 \to (L_2')^1$ is the identity on the defined boundaries $\vee^n\SS^1$ and, by inclusion, gives a map $g: (L_1')^1 \to L_2'$. If $\phi : S^1 \to (L_1')^1$ denotes the attaching map for a relator, then $g \circ \phi$ is nullhomotopic since the presentations both present $G$. Hence, by basic algebraic topology, $g$ extends to a map $f : L_1' \to L_2'$. 
This restricts to the identification $(L_1')^1 \to (L_2')^1$ and, since the generators are mapped to generators, $\pi_1(f)$ is an isomorphism. 
\end{proof}

The following allows us to compute $\cdot \ : \ \mcK_n\times \mcK_n\lra \mcK_n$ for the 2-complexes arising from Lemma \ref{lem: can fix boundary}. 
If $L \in \mcK_n$, then define $L \vee \SS^2 \in \mcK_n$ to be the 2-complex whose boundary is identified with $\vee^n \SS^1$ via the identification $L^1 \cong (L \vee \SS^2)^1$ induced by inclusion.

\begin{lemma} \label{lem: L_1.L_2}
Let $L_1, L_2\in \mcK_n$ be $2$-complexes such that there is a map $f\colon L_1 \to L_2$ fixing their common boundary $\vee^n\SS^1 = L_1^1= L_2^1$ and inducing an isomorphism on $\pi_1$. If $l,m$ denote the number of $2$-cells of $L_1, L_2$ respectively, then
$$L_1\cdot L_2 = L_1 \vee^m \SS^2 = L_2 \vee^l \SS^2 \in \mcK_n.$$
\end{lemma}

\begin{proof}[Proof of Lemma \ref{lem: L_1.L_2}]
It suffices to prove that $L_1\cdot L_2 = L_1 \vee^m \SS^2 \in \mcK_n$.

Let $G = \pi_1(L_1) \cong \pi_1(L_2)$. By hypothesis, $L_1$ and $L_2$ correspond to presentations on the same generating set for $G$. We will denote the presentations respectively by
$$\langle x_1,\ldots, x_n \, |\,  R_1, \ldots, R_l\rangle\; ,\;   \langle x_1,\ldots, x_n \, | \, S_1, \ldots, S_m\rangle. $$
Then $L_1 \cdot L_2$ corresponds to
$$\langle x_1,\ldots, x_n \, |\,  R_1, \ldots, R_l, S_1, \ldots, S_m\rangle. $$

Note that $G \cong F_{x_1,\ldots,x_n}/N$ where $N = N(R_1,\ldots,R_l)$ denotes the normal closure of relators $R_1, \ldots, R_l$. Since each $S_i$ is trivial in $G$, we have that $S_i \in N$ and so 
\[ S_i = (g_1^{-1}r_1g_1) \ldots (g_t^{-1}r_t g_t) \in F_{x_1,\dots,x_n} \]
for some $g_i \in F_{x_1,\dots,x_n}$ and $r_i \in \{R_1^{\pm 1},\ldots,R_l^{\pm 1}\}$.

We now claim that, in the presentation for $L_1 \cdot L_2$, we can replace each $S_i$ with the trivial relator $1$ by a sequence of 3-deformations which fix the 1-skeleton. 
Indeed, the change $ S_i \mapsto (g_1^{-1}r_1g_1)^{-1}S_i$ is implemented by Tietze moves of type (i). Repeating this for successive $g_j^{-1}r_jg_j$ gives $S_i \mapsto 1$, showing
that the presentation for $L_1 \cdot L_2$ is equivalent, by 3-deformations fixing the 1-skeleton, to
$$\langle x_1,\ldots, x_n \, |\,  R_1, \ldots, R_l, 1, \ldots, 1\rangle $$
which corresponds to $L_1 \vee^m \SS^2$. 
\end{proof}

Now suppose that $L_1, L_2$ are $2$-complexes such that $\pi_1(L_1) \cong \pi_1(L_2)$ and $\chi(L_1)=\chi(L_2)$.
By Lemma \ref{lem: can fix boundary}, there exist $L_1',L_2' \in \mcK_n$ such that $L_i$ is 3-deformation equivalent to $L_i'$ for each $i=1,2$ and there is a map $f: L_1' \to L_2'$ fixing their common boundary and such that $\pi_1(f)$ is an isomorphism. Since $\chi(L_1)=\chi(L_2)$ (and so $\chi(L_1')=\chi(L_2')$), the number of 2-cells each of $L_1',L_2'$ has is $m=\chi(L_1)-1+n$.
By Lemma \ref{lem: L_1.L_2},
\[ L_1' \cdot L_1' = L_1' \vee^m \SS^2 = L_1' \cdot L_2' = L_2' \cdot L_1' = L_2' \vee^m \SS^2 = L_2' \cdot L_2' \in \mcK_n. \]

In particular, if $x = L_1'-L_2'$, then 
\[ x \cdot x = L_1' \cdot L_1' -L_1' \cdot L_2' - L_2' \cdot L_1' + L_2' \cdot L_2' = 0 \in \kk \mcK_n \]
as required.
This completes the proof of Theorem \ref{thm: simple homotopy versus 3 deformations}.
\end{proof}

\begin{remark} \label{rem: after proof}
\hfill
\begin{enumerate}[label=(\roman*)]
\item
Combining Lemmas \ref{lem: can fix boundary} and \ref{lem: L_1.L_2}
also implies that, if $L_1, L_2$ are $2$-complexes with $\pi_1(L_1) \cong \pi_1(L_2)$ and $l,m$ are their numbers of $2$-cells respectively, then $L_1 \vee^m \SS^2$ and $L_2 \vee^l \SS^2$ are $3$-deformation equivalent.
This result, which follows from the Tietze theorem, was mentioned in Section \ref{sec: background}. The argument is essentially the same as the classical one. 
\item
For our later applications, we apply Theorem \ref{thm: simple homotopy versus 3 deformations} to pairs of 2-complexes $L_1$ and $L_2$ with $\pi_1(L_1) \cong \pi_1(L_2)$ and $\chi(L_1)=\chi(L_2)$ but which are not 3-deformation equivalent. Such examples are described in Section \ref{sec: background}.  
If the examples $L_1$ and $L_2$ we use are not homotopy equivalent (as in the (1) vs. (2) examples of Section \ref{sec: background}), then a slightly weaker version of Theorem \ref{thm: simple homotopy versus 3 deformations} would suffice. Namely, the $3$-deformation equivalences in item (i) of the theorem need only be homotopy equivalences. In this case, we can use \cite[Proposition 2.1]{Johnson1} in place of Lemma \ref{lem: can fix boundary} and thus obtain a simpler proof.
\end{enumerate}
\end{remark}

\begin{cor} \label{cor: null vector not homotopic}
For each $n\geq 3$ there exist non-trivial elements $x\in \kk \mcK_n$ such that $x\cdot x=0$. 
\end{cor}

\begin{proof}
To deduce this corollary from Theorem \ref{thm: simple homotopy versus 3 deformations},
it suffices for each $n\geq 3$ to exhibit 2-complexes $L_1$ and $L_2$ such that $\pi_1(L_1) \cong \pi_1(L_2)$ and $\chi(L_1) = \chi(L_2)$ which are not 3-deformation equivalent, and so that the corresponding $L'_1, L'_2$ are elements of $\kk \mcK_n$. In the hierarchy of equivalence relations of 2-complexes discussed in Section \ref{sec: background}, such examples should be equivalent by (1) but not (4). 

By the discussion in Section \ref{sec: background}, examples of $L_1$ and $L_2$ which are simply connected would require a counterexample to the Andrews-Curtis conjecture. Examples in the non-simply connected case are listed at the end of Section \ref{sec: background} and are of two types. They could be stably homotopy equivalent but not homotopy equivalent (i.e. (1) but not (2)), or homotopy equivalent but not simple homotopy equivalent (i.e. (2) but not (3)).

To be specific, consider $n=3$ and the examples in \cite{Lustig} showing that the standard $2$-complexes $K_i$ for the group presentations
$$ K_i:=\langle   r, s, t\, |\,  s^2=t^3, [r^2, s^{2i+1}]= [r^2, t^{3i+1}]=1
\rangle,$$
$i=1, 2,\ldots$, have isomorphic $\pi_1$ and the same Euler characteristic, but are pairwise homotopy inequivalent. For any $i, j$ the maps $K_i\to K_j$ inducing an isomorphism on $\pi_1$ are obtained by mapping $\vee^3 \SS^1$ by the identity and extending to the $2$-cells. (Lemma \ref{lem: can fix boundary} in the proof of Theorem \ref{thm: simple homotopy versus 3 deformations} was used to find maps fixing the boundary, and the preceding sentence indicates that for these examples, the maps already satisfy this property.)
Therefore the $2$-complexes $K_i$ are all distinct elements in $\kk {\mathcal K}_3$, where $\vee^3 \SS^1$ correspond to the generators $r,s,t$. 

This concludes the proof of the corollary for $n=3$; it follows for $n>3$ by taking a wedge sum of the $2$-complexes involved in the construction with $\vee^{n-3} \SS^1$.

\end{proof}

\begin{remark}
\hfill
\begin{enumerate}[label=(\roman*)]
\item 
Corollary \ref{cor: null vector not homotopic} constructs an infinite collection $\{ K_i \}_{i\ge 1}$  of elements in $\kk \mcK_n$. Let \[\kk \mcK_n' \ := \ \kk\mcK_n/\ker (\langle,\rangle)
\]
be the quotient of $\kk\mcK_n$ by the kernel of the bilinear form \eqref{eq: universal pairing on Kn}. It is an interesting question whether the elements $K_i-K_{i+1}$ are $\kk$-linearly independent in $\kk \mcK_n'$ over all $i\ge 1$.
\item 
We point out that for any $n$ there does not exist a non-zero $x\in \kk \mcK_n$ such that $x\cdot y=0\in\kk\mcK_n$ for all $y\in \kk \mcK_n$. This is due to the fact that $\kk \mcK_n$
    has a unit, $\vee^n \SS^1$. 
    The question whether there exists a non-trivial $x$ with $\langle x, y \rangle=0$ for any $y$, 
    analogous to the problem in the context of the universal pairing for $4$-manifolds \cite[Problem 2]{FKNSWW}), is open.
\end{enumerate}
\end{remark}

\section{Simple homotopy equivalent 2-complexes} \label{sec: simple homotopy equivalent}

Theorem \ref{thm: simple homotopy versus 3 deformations}  showed that the universal pairing does not detect the difference between stable equivalence and $3$-deformations for $2$-complexes.
To analyze simple homotopy equivalence, in this section we define a version of the
universal pairing using an equivalence relation that is finer than 3-deformation. 

Recall that elementary $2$-expansions and collapses of $2$-complexes correspond to Tietze moves (ii), (iii) listed in Section \ref{sec: background}. They change the identification of the boundary of $2$-complexes with  $\vee^n\SS^1$ and thus affect the element represented by a $2$-complex in $\mcK_n$. The definition of $\mcK_n$ involved $2$-complexes up to $3$-deformations fixing $\vee^n\SS^1$, that is the moves on $2$-complexes corresponding to the Tietze moves
\begin{equation} \label{eq: 3-def}
R_j\mapsto w R_j w^{-1}, \, R_j\mapsto R_j^{-1}, \, R_j\mapsto R_j R_k \; {\rm or} \; R_j\mapsto R_k R_j\;  {\rm where} \; w\in F_{g_1, \ldots, g_m}, \, j\neq k.
\end{equation}
Consider a restricted version, where in addition to the first two moves in \eqref{eq: 3-def} we have 
\begin{equation} \label{eq: 3-def alg zero}
R_j\mapsto R_j \cdot \prod_i w_i [R^{\pm 1}_{k_i}, h_i]w_i^{-1} \; 
{\rm where} \; w_i, h_i\in F_{g_1, \ldots, g_m}, \; {\rm and} \; k_i\neq j \; {\rm for}\; {\rm each}\; i.
\end{equation}
In other words, only compositions of handle slides of $R_j$ over $R_k$ in \eqref{eq: 3-def} are allowed, $k\neq j$, where the total exponent of $R_k$ is zero. 
Define $\mcK'_n$ to be $2$-complexes with $1$-skeleton identified with $\vee^n\SS^1$, modulo the first two moves in \eqref{eq: 3-def} and the moves \eqref{eq: 3-def alg zero}. This equivalence relation fits in between (4) and (5) in the list in Section \ref{sec: background}.

As discussed below, Quinn \cite{Quinn} showed that a simple homotopy equivalence is a composition of a $3$-deformation, an {\em $s$-move}, and another $3$-deformation. 
Schematically, 
\[
L_1 \stackrel{\mathrm{2-exp}}{\nearrow} L'_1 \stackrel{\mathrm{3-exp}}{\nearrow}
L''_1
\xrsquigarrow{\mathrm{s-moves}}
L''_2 \stackrel{\mathrm{3-col}}{\searrow}
L_2'\stackrel{\mathrm{2-col}}{\searrow} L_2
\]
Here $n$-exp, respectively $n$-col stands for $n$-expansion, respectively $n$-collapse. 
Therefore the key question for the Andrews-Curtis conjecture is whether the $s$-move can be expressed as a $3$-deformation. The following result shows that even with the finer equivalence relation defining $\mcK'_n$, the universal pairing cannot detect the difference.

\begin{theorem}\label{thm: simple homotopy versus 3 deformations+} 
Let $L_1, L_2\in \mcK'_n$ be two $2$-complexes
related by an $s$-move. Then $x:=L_1-L_2\in \kk \mcK'_n$ satisfies $x\cdot x=0$, so in particular  $\langle x, x\rangle=0$.
\end{theorem}

{\em Proof of Theorem \ref{thm: simple homotopy versus 3 deformations+}}
We start by recalling the definition of the {\em s-move} defined by Quinn \cite[2.1]{Quinn}. Consider the family of $2$-complexes assembled of surfaces $\Sigma=\coprod_i \Sigma_i$ and annuli $A=\coprod_{i,j} A_{i,j}, B=\coprod_{i,j} B_{i,j}$. The data for this move is the following:
\begin{itemize} \label{s move data}
    \item Compact connected orientable surfaces  $\Sigma_i$, $1\leq i\leq m$, each one with two boundary components denoted $R=\coprod_i R_i, S=\coprod_i S_i$.
    \item A symplectic basis of simple closed curves $\{a_{i,j}, b_{i,j} \}$ on each surface $\Sigma_i$, $1\leq j\leq {\rm genus}(\Sigma_i)$.
    \item Annuli $A_{i,j}, B_{i,j}$ with $\partial A_{i,j}=a_{i,j}\coprod R_k, \partial B_{i,j}=b_{i,j}\coprod S_l$ for some $k,l$ depending on $i,j$. 
\end{itemize}

\begin{definition}\label{def: s move}
    Two $2$-complexes $L_1, L_2$ are related by a Quinn $s$-move if there exist a $2$-complex $K$, surfaces $\Sigma$ and annuli $A, B$ as above, and a map $f\colon \Sigma\cup A\cup B \longrightarrow K$ with 
    $$ L_1\, =\, K\, \bigcup_{f(R)}\,  mD^2, \; \,  L_2\, =\, K\, \bigcup_{f(S)}\,  mD^2.$$
\end{definition}

Note that homotopies of the attaching maps of $2$-cells are $3$-deformations, cf. \cite[Lemma 2.1]{HM1} and the proof below, so $R,S$ may be assumed to map to the $1$-skeleton of $K$. 
This is an implicit assumption in the above definition, so that $mD^2$ are attached to the $1$-skeleton. On the other hand, the curves $f(a_{i,j}), f(b_{i,j})$ are not assumed to be in the $1$-skeleton.

The data for the $s$-move in the genus $1$ case and $m=1$ is illustrated in Figures \ref{figure:Quinn move}, \ref{figure:Quinn move2}.
(The boundary curves $R, S$ of the surface are drawn as based curves. Generally, a homotopy of the surface $\Sigma$ into this position results in conjugation; this is not shown in the figure.) 
A more elaborate example is given in \cite{Quinn}.

\begin{figure}[ht]
\centering
\includegraphics[height=2.3cm]{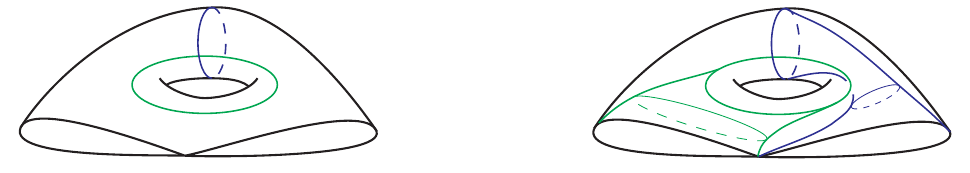}
{\scriptsize
\put(-365,2){$R$}
\put(-227,2){$S$}
\put(-321,25){$a$}
\put(-298,50){$b$}
\put(-148,2){$R$}
\put(-10,2){$S$}
\put(-115,37){$A$}
\put(-60,45){$B$}
}
\caption{The setting for a Quinn  $s$-move in the smallest non-trivial example: a connected, genus $1$ surface $\Sigma$ with boundary $R\cup S$ and a symplectic basis of curves $a, b$. Right: the annulus $A$ is attached to $\Sigma$ along the curves $a, R$ and the annulus $B$ is attached to $\Sigma$ along $b, S$. The union $\Sigma\cup A\cup B$ is mapped to some $2$-complex $K$.}
\label{figure:Quinn move}
\end{figure}

\begin{figure}[ht]
\centering
\includegraphics[height=2.8cm]{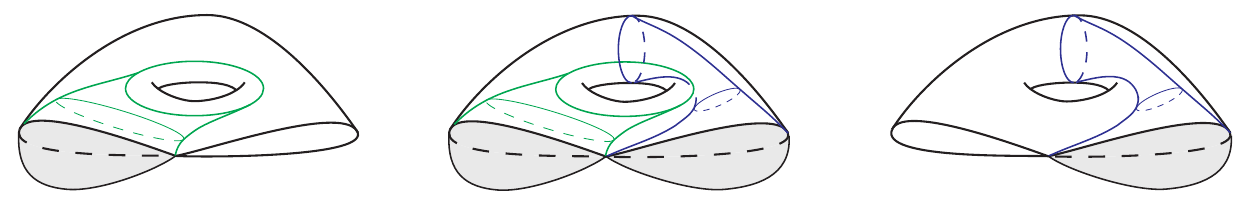}
{\scriptsize
\put(-445,5){$D_R$}
\put(-293,5){$D_R$}
\put(-453,30){$R$}
\put(-316,30){$S$}
%\put(-430,65){$L_1$}
\put(-395,58){$a$}
%\put(-292,65){$L_1\cdot L_2$}
\put(-176,5){$D_S$}
\put(-15,5){$D_S$}
\put(-3,30){$S$}
\put(-140,30){$R$}
%\put(-117,65){$L_2$}
\put(-74,61){$b$}
}
\caption{Left: a null-homotopy for $RS^{-1}$ in $\Sigma\cup A\cup D_R$ is provided by the surface $\Sigma$ surgered along the disk $A\cup D_R$ attached to $a$. (The curve $RS^{-1}$ is defined using the induced orientation on the boundary of the surface $\Sigma$.) 
Right: a null-homotopy for $RS^{-1}$ in $\Sigma\cup B\cup D_S$ is given by $\Sigma$ surgered along the disk  $B\cup D_S$ attached to $b$.}
\label{figure:Quinn move2}
\end{figure}
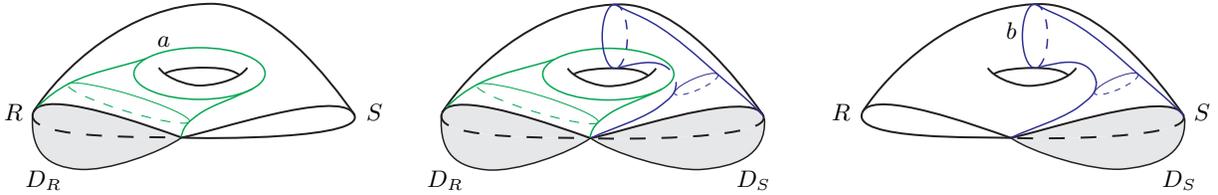

It follows from the definition that $2$-complexes related by an $s$-move are simple homotopy equivalent. In fact, they are related by a $4$-deformation, see \cite[2.4]{Quinn}.

The following result is an analogue of Lemma \ref{lem: L_1.L_2} for $\mcK'_n$; its proof in this case is quite different.
In the lemma below, as in Theorem \ref{thm: simple homotopy versus 3 deformations+}, the entire $1$-skeleton of the $2$-complexes is identified with $\vee^n\SS^1$.

\begin{lemma} \label{lem: s move}
    Let $L_1, L_2\in \mcK'_n$ be $2$-complexes related by an $s$-move. Then $$L_1\cdot L_1=L_1\cdot L_2=L_2\cdot L_2\in \mcK'_n.$$
\end{lemma}

\begin{proof}[Proof of Lemma \ref{lem: s move}]
We will show that $L_1\cdot L_2$ is equivalent to $L_1\cdot L_1$ with respect to the moves \eqref{eq: 3-def alg zero}. The equivalence with $L_2\cdot L_2$ will follow by a directly analogous argument. In the notation of Definition \ref{def: s move}, $L_1\cdot L_2$ is obtained from $K\cdot K$ by attaching two collections of $2$-cells: $m$ disks attached to $R$, and $m$ disks attached to $S$. Denote these two collections of disks by $D_R, D_S$ respectively. 

The $2$-complex $L_1=K\cup D_R$ is a subcomplex of $L_1\cdot L_2$. 
Consider $$f(\Sigma\cup A)\cup D_R\, \subset\,  K\cup D_R\, \subset\,  L_1\cdot L_2.$$
The annuli $A$ provide a free homotopy between the curves $R$ and half a symplectic basis of curves, $a$, in $\Sigma$. Thus the attaching curves $S$ for $D_S$ can be transformed to the attaching curves $R$ for $D_R$ in $K\cup D_R$ using moves \eqref{eq: 3-def alg zero}. 
In other words, $L_1\cup_S D_S=L_1\cup\,$(a collection of disks attached along $R$) $\in \mcK'_n$.  Note that
$L_1\cdot L_2$ and $L_1\cdot L_1$ are obtained from these two $2$-complexes by adding the second copy of $2$-cells of $K$.
This concludes the proof of Lemma \ref{lem: s move} and of Theorem~\ref{thm: simple homotopy versus 3 deformations+}.
\end{proof}

\begin{remark}
The proof above was given in the context of $2$-complexes; an equivalent proof may be phrased using group presentations.
In fact, an algebraic extension of the $s$-move to derived series was established by Hog-Angeloni and Metzler \cite{HM2}. Using these methods, a version of Theorem \ref{thm: simple homotopy versus 3 deformations+} can be proved for a derived series analogue of the relations \eqref{eq: 3-def alg zero}.

\end{remark}

\section{The universal pairing for higher dimensional complexes} \label{sec: higher dimensions}

So far, we have considered the universal pairing for 2-complexes, but it can be defined over complexes in arbitrary dimensions. The following generalizes Definition \ref{def: kn}. In this section, we will use $n$ to denote the dimension of the complexes.

\begin{definition}
Let $n \ge 2$ and let $L$ be an $(n-1)$-complex. Then define $\mcK_L$ to be the set of equivalence classes of $n$-complexes K with a subcomplex in the $(n-1)$-skeleton of $K$ identified with $L$, considered up to $(n+1)$-deformations restricting to the identity on $L$. Given $K \in \mcK_L$, we refer to $L$ as the \textit{boundary} of $K$.
\end{definition}

As before, $\mcK_L$ comes with a commutative associative multiplication 
\begin{equation}\label{eq_mult_K2}\cdot \ : \ \mcK_L\times \mcK_L\lra \mcK_L
\end{equation}
given by taking the union of two $n$-complexes along the common boundary $L$. The $n$-complex $L$ without $n$-cells is the unit element for  multiplication.

For a commutative ring $\kk$, we obtain the multiplication map 
$\cdot\,\; \colon \; \kk \mcK_L\times \kk \mcK_L\longrightarrow \kk \mcK_L$ and the bilinear pairing $\langle\, , \, \rangle\,\; \colon \; \kk \mcK_L\times \kk \mcK_L\longrightarrow \kk \mcK_\ast$, analogously to the case $n=2$ discussed in Section \ref{sec: universal construction for 2-complexes}.

One key difference in the case $n \ge 3$ is that, as mentioned in the instruction, $(n+1)$-deformation equivalence corresponds precisely to simple homotopy equivalence \cite{Wa1}. In particular, for $n \ge 3$, $\mcK_L$ is the set of $m$-complexes with boundary $L$ up to simple homotopy equivalences restricting to the identity on $L$.

The following result is the analogue of Corollary \ref{cor: null vector not homotopic}.

\begin{theorem} \label{thm: null vector not homotopic - higher dim}
For each $n \ge 3$, there exists an $(n-1)$-complex $L$ and a non-trivial element $x \in \kk \mcK_L$ such that $x \cdot x=0$.  
\end{theorem}

To achieve this, we will focus on a special class of $n$-complexes which are more amenable to computations of this sort. 
Recall that, for a group $G$, a \textit{$(G,n)$-complex} is an $n$-complex $X$ with the $(n-1)$-type of the Eilenberg-Maclane space $K(G,1)$. That is, $X$ is an $n$-complex such that $\pi_1(X) \cong G$ and $\pi_i(X)=0$ for $2 \le i \le n-1$. 
Similarly to the case of 2-complexes (see Section \ref{sec: background}), we can ask when there exist $(G,n)$-complexes $L_1$ and $L_2$ such that $\chi(L_1)=\chi(L_2)$ but which are not homotopy equivalent. Examples are known in the following cases. 
\begin{itemize}
\item 
For finite groups and all $n \ge 3$, examples are known for certain abelian groups by Sieradski-Dyer \cite{SD79} and certain groups with periodic cohomology by Dyer \cite{Dy76} and Nicholson \cite{N20}.
\item
For infinite groups with finite cohomological dimension and all $n \ge 3$, examples were constructed by Harlander-Jenson \cite{HJ2}, and Nicholson \cite{N21} constructed examples with arbitrary Euler characteristic away from the optimal value.
\end{itemize}

From this point onwards, one could 
aim to generalize Lemmas \ref{lem: can fix boundary} and \ref{lem: L_1.L_2} from 2-complexes for $(G,n)$-complexes in order to obtain a result analogous to Theorem \ref{thm: simple homotopy versus 3 deformations}. 
Instead, focusing on the proof of Theorem \ref{thm: null vector not homotopic - higher dim}, we will demonstrate a simple approach which works in a special case. 

\begin{theorem} \label{thm: simple homotopy special case}
Let $n \ge 3$, let $G$ be a finite group and let $L_1$ and $L_2$ be $(G,n)$-complexes which are not homotopy equivalent but $\chi(L_1)=\chi(L_2)$, $L_1^{(n-1)}=L_2^{(n-1)}$ and $L_1$, $L_2$ have at least two $n$-cells.
Fix identifications $L=L_1^{(n-1)}=L_2^{(n-1)}$ so that $L_1, L_2 \in \mcK_L$. 

Then $x:=L_1-L_2 \in \kk \mcK_L$ satisfies $x\cdot x=0$, so in particular $x$ is a null vector for the universal pairing: $\langle x, x\rangle=0$.    
\end{theorem}

The following shows that the hypothesis of Theorem \ref{thm: simple homotopy special case} can be satisfied.

\begin{lemma} \label{lem: special examples}
For each $n \ge 3$, there exist a finite group $G$ and $(G,n)$-complexes $L_1$ and $L_2$ which are not homotopy equivalent but $\chi(L_1)=\chi(L_2)$, $L_1^{(n-1)}=L_2^{(n-1)}$ and $L_1$, $L_2$ have at least two $n$-cells.
\end{lemma}

\begin{proof}[Proof of Lemma \ref{lem: special examples}]
 This follows from the examples of Sieradski-Dyer \cite{SD79} in the case where $G$ is finite abelian. In fact, for a fixed finite group $G$, all $(G,n)$-complexes constructed in the article have the same $(n-1)$-skeleta and at least two $n$-cells \cite[Proof of Proposition 6]{SD79}. The existence of examples for each $n \ge 3$ follows by substituting values into \cite[Proposition 8]{SD79}.
Examples of Dyer \cite{Dy76} and Nicholson \cite{N20} also imply Lemma \ref{lem: special examples} in the case of certain finite groups with periodic cohomology.    
\end{proof}

We now turn to the proof of Theorem \ref{thm: simple homotopy special case}. The argument we give requires having dimensions $n \ge 3$ and so is fundamentally different to the proof of Theorem \ref{thm: simple homotopy versus 3 deformations}.

We will make use of the following general fact, which applies for all $n$ and $G$.

\begin{lemma} \label{lem: product (G,n)-complexes}
Let $n \ge 2$, let $G$ be a group and let $L_1, L_2$ be $(G,n)$-complexes equipped with identifications $L=L_1^{(n-1)}=L_2^{(n-1)}$ so that $L_1,L_2 \in \mcK_L$. Then $L_1 \cdot L_2$ is a $(G,n)$-complex.
\end{lemma}

\begin{proof}[Proof of Lemma \ref{lem: product (G,n)-complexes}]
Let $K=L_1 \cdot L_2$.
Since $K$ and $L_1$ have the same $(n-1)$-skeleton, it follows that $\pi_i(K)=0$ for $2 \le i \le n-2$. Let $\widetilde{K}$ denote the universal cover of $K$. By the Hurewicz theorem and standard facts about homotopy groups, we have $\pi_{n-1}(K) \cong \pi_{n-1}(\widetilde{K}) \cong H_{n-1}(\widetilde{K})$ and so it suffices to check that $H_{n-1}(\widetilde{K})=0$. By construction, the cellular chain complex $C_*(\widetilde{K})$ has the form
\[
\begin{tikzcd}[row sep=tiny]
C_n(\widetilde{L_1}) \arrow[dr,"\partial_n^{L_1}"] & &&&& \\
\oplus & C_{n-1}(\widetilde{L}) \arrow[r,"\partial_{n-1}"] & \cdots \arrow[r,"\partial_{2}"] & C_{1}(\widetilde{L})  \arrow[r,"\partial_{1}"] & C_{0}(\widetilde{L}) \\
C_n(\widetilde{L_2}) \arrow[ur,"\partial_n^{L_2}"'] & &&&&
\end{tikzcd}
\]
where, for $i=1,2$, the sequence $(\partial_n^{L_i},\partial_{n-1},\cdots,\partial_1)$ corresponds to $C_*(\widetilde{L_i})$.

For $i=1,2$, since $L_i$ is a $(G,n)$-complex, we have that $H_{n-1}(\widetilde{L_i}) \cong \pi_{n-1}(L_i)=0$ and so $\text{im}(\partial_n^{L_i})=\ker(\partial_{n-1})$. It follows that $\text{im}(\partial_n^{L_1}) = \text{im}(\partial_n^{L_2})$ and so
\[ \text{im}(\partial_n^{L_1},\partial_n^{L_2}) = \text{im}(\partial_n^{L_1}) + \text{im}(\partial_n^{L_2}) = \text{im}(\partial_n^{L_1}) = \ker(\partial_{n-1}) \]
which implies that $H_{n-1}(\widetilde{K}) = 0$, as required.    
\end{proof}

The following generalizes a result of \cite[Theorem 3]{Dy81} to the relative case.

\begin{lemma} \label{lem: Dyer+}
Let $n \ge 2$, let $G$ be a finite group and let $L_1, L_2$ be $(G,n)$-complexes with identifications $L=L_1^{(n-1)}=L_2^{(n-1)}$ and $(-1)^n\chi(L_1) = (-1)^n\chi(L_2) \ge 2 + (-1)^n \chi(L_0)$ for some $(G,n)$-complex $L_0$. Then $L_1$ and $L_2$ are simple homotopy equivalent by a map which restricts to the identity on $L$.    
\end{lemma}

\begin{proof}
We start by establishing the result for homotopy equivalences.
By \cite[Theorem 3]{Dy81}, $L_1$ and $L_2$ are homotopy equivalent. In particular, $C_*(\widetilde{L_1})$ and $C_*(\widetilde{L_2})$ are chain homotopy equivalent. Since $C_{* \le n-1}(\widetilde{L_1})=C_{* \le n-1}(\widetilde{L_2}) = C_*(\widetilde{L})$, we can apply \cite[Theorem 8.2]{Johnson2} to get there exists a chain homotopy equivalence $F : C_*(\widetilde{L_1}) \to C_*(\widetilde{L_2})$ which restricts to the identity on $C_*(\widetilde{L})$.
For $(G,n)$-complexes, all maps on cellular chain complexes are geometrically realisable (see, for example, \cite{Dy76}) and so there exists a map $f : L_1 \to L_2$ which restricts to the identity on $L$ and is such that $C_*(f) = F$. Since $C_*(f)$ is a chain homotopy equivalence, $f$ is a homotopy equivalence, as required.

Next, \cite[Theorem 2]{Dy81} implies that there exists a self homotopy equivalence $g : L_2 \to L_2$ such that $\tau(g) = \tau(f) \in \text{Wh}(\pi_1(L_2))$, where $\tau$ denotes the Whitehead torsion and Wh denotes the Whitehead group. By the same argument as above, \cite[Proof of Theorem 2]{Dy81} implies that we can assume $g$ restricts to the identity on $L$. If $\bar{g}$ is the homotopy inverse of $f$, then $\bar{g} \circ f : L_1 \to L_2$ is a homotopy equivalence restricting to the identity on $L$ and which is such that $\tau(\bar{g} \circ f) = 0 \in \text{Wh}(\pi_1(L_2))$, i.e. $\bar{g} \circ f$ is a simple homotopy equivalence.
\end{proof}

\begin{proof}[Proof of Theorem \ref{thm: simple homotopy special case}]
Let $m \ge 2$ denote the number of $n$-cells of $L_1$. This is also the number of $n$-cells of $L_2$ since $\chi(L_1)=\chi(L_2)$ and $L_1^{(n-1)}=L_2^{(n-1)}$.
Let $i,j \in \{1,2\}$. By Lemma \ref{lem: product (G,n)-complexes}, 
$L_i \cdot L_j$ is a $(G,n)$-complex.
By construction, we have that 
\begin{align*} 
(-1)^n\chi(L_i \cdot L_j) &= (-1)^n\chi(L_i) + \#\{\text{$n$-cells of $L_j$}\} \\
&= (-1)^n\chi(L_1) + m \ge (-1)^n\chi(L_1) + 2.
\end{align*}
Hence, by Lemma \ref{lem: Dyer+}, $L_1 \cdot L_2$, $L_1 \cdot L_2$ and $L_2 \cdot L_2$ are each simple homotopy equivalent by maps which fix the common $(n-1)$-skeleton $L$. Since $n \ge 3$, 
\cite[Theorem 1]{Wa1} implies they are equivalent by $(n+1)$-deformations fixing $L$.  
\end{proof}

Combining Theorem \ref{thm: simple homotopy special case} with Lemma \ref{lem: special examples} completes the proof of Theorem \ref{thm: null vector not homotopic - higher dim}.

\begin{remark}
It would be interesting to see if a version of Theorem \ref{thm: simple homotopy special case} holds without the assumption that $G$ is a finite group. However, the key obstacle to obtaining such a generalization is that it is not currently clear whether \cite[Theorem 2]{Dy81} has an analogue over arbitrary finitely presented groups.
\end{remark}

%\addcontentsline{toc}{section}{References}
%\def\refname{}

\end{document}